\documentclass[graybox]{svmult}
\usepackage{mathptmx}
\usepackage{helvet}
\usepackage{courier}
\usepackage[latin9]{inputenc}
\usepackage{xcolor}
\usepackage{url}
\usepackage{amsmath}
\usepackage{amssymb}
\usepackage{esint}
\PassOptionsToPackage{normalem}{ulem}
\usepackage{ulem}

\makeatletter

\providecolor{lyxadded}{rgb}{0,0,1}
\providecolor{lyxdeleted}{rgb}{1,0,0}

\usepackage{enumitem}		
\newlength{\lyxlistindent}      
\usepackage{graphicx}
\usepackage{multicol}
\usepackage[bottom]{footmisc}
\usepackage{amsopn}

\setlist[enumerate]{leftmargin=*,label=(\roman*),align=left}

\newcommand{\xyR}[1]{ \makeatletter
\xydef@\xymatrixrowsep@{#1} \makeatother} 
\newcommand{\xyC}[1]{ \makeatletter
\xydef@\xymatrixcolsep@{#1} \makeatother} 

\newcommand{\ra}{\longrightarrow}

\newcommand{\field}[1]{\mathbb{#1}}
\newcommand{\R}{\field{R}} 
\newcommand{\N}{\field{N}} 
\newcommand{\Z}{\ensuremath{\mathbb{Z}}} 

\newcommand{\eps}{\varepsilon} 
\renewcommand{\phi}{\varphi}
\newcommand{\diff}[1]{\,\hbox{\rm d}#1} 

\newcommand{\Coo}{\mbox{\ensuremath{\mathcal{C}}}^{\infty}} 

\newcommand{\Rtil}{\widetilde \R} 
\newcommand{\otilc}{\widetilde \Omega_c} 
\newcommand{\gs}{\mathcal{G}^s} 
\newcommand{\ns}{\mathcal{N}^s} 

\newcommand{\sint}[1]{\langle#1\rangle} 
\newcommand{\Eball}{B^{{\scriptscriptstyle \text{\rm E}}}} 

\newcommand{\RC}[1]{{}^{\scriptscriptstyle #1}\Rtil}
\newcommand{\csp}[1]{{}^{\scriptscriptstyle \rho}\widetilde{#1_{\text{\rm c}}}}

\newcommand{\supp}{\mbox{supp}}
\newcommand{\cinfty}{{\mathcal C}^\infty}

\newcommand{\comp}{\Subset}
\newcommand{\esm}{{\mathcal E}_M}

\newcommand{\Om}{\Omega}

\newcommand{\sse}{\subseteq}

\newcommand{\rti}{\RC{\rho}}
\newcommand{\gsf}{{}^{\scriptscriptstyle \rho}{\mathcal{GC}}^\infty}
\newcommand{\gsfo}{{\mathcal{GC}}^\infty}
\newcommand{\no}[1]{| #1|}

\makeatother

\begin{document}

\title*{Inverse Function Theorems for Generalized Smooth Functions}

\author{Paolo Giordano and Michael Kunzinger}

\institute{Paolo Giordano \at University of Vienna, Oskar-Morgenstern-Platz
1, 1090 Wien \email{paolo.giordano@univie.ac.at} \and Michael Kunzinger
\at University of Vienna, Oskar-Morgenstern-Platz 1, 1090 Wien \email{michael.kunzinger@univie.ac.at}}
\maketitle

\abstract*{Generalized smooth functions are a possible formalization of the
original historical approach followed by Cauchy, Poisson, Kirchhoff,
Helmholtz, Kelvin, Heaviside, and Dirac to deal with generalized functions.
They are set-theoretical functions defined on a natural non-Archimedean
ring, and include Colombeau generalized functions (and hence also
Schwartz distributions) as a particular case. One of their key property
is the closure with respect to composition. We review the theory of
generalized smooth functions and prove both the local and some global
inverse function theorems.}

\section{Introduction\label{sec:intro}}

Since its inception, category theory has underscored the importance
of unrestricted composition of morphisms for many parts of mathematics.
The closure of a given space of ``arrows'' with respect to composition
proved to be a key foundational property. It is therefore clear that
the lack of this feature for Schwartz distributions has considerable
consequences in the study of differential equations \cite{Kan98,ErlGross},
in mathematical physics \cite{Ba-Zo90,Bo-Sh59,Col92,Dir26,Gi-Ku-St15,Gon-Sca56,Gsp09,Mar68,Mar69,Mik66,Raj82},
and in the calculus of variations \cite{KoKuMO08}, to name but a
few.

On the other hand, Schwartz distributions are so deeply rooted in
the linear framework that one can even isomorphically approach them
focusing only on this aspect, opting for a completely formal/syntactic
viewpoint and without requiring any functional analysis, see \cite{Seb54}.
So, Schwartz distributions do not have a notion of pointwise evaluation
in general, and do not form a category, although it is well known
that certain subclasses of distributions have meaningful notions of
pointwise evaluation, see e.g.~\cite{Loj57,Loj58,Rob66,Pee68,CaFe01,EsVi12,VeVi12}.

This is even more surprising if one takes into account the earlier
historical genesis of generalized functions dating back to authors
like Cauchy, Poisson, Kirchhoff, Helmholtz, Kelvin, Heaviside, and
Dirac, see \cite{Kat-Tal12,Lau89,Lau92,vdP-Bre87}. For them, this
``generalization'' is simply accomplished by fixing an infinitesimal
or infinite parameter in an ordinary smooth function, e.g.~an infinitesimal
and invertible standard deviation in a Gaussian probability density.
Therefore, generalized functions are thought of as some kind of smooth
set-theoretical functions defined and valued in a suitable non-Archimedean
ring of scalars. From this intuitive point of view, they clearly have
point values and form a category.

This aspect also bears upon the concept of (a generalized) solution
of a differential equation. In fact, any theory of generalized functions
must have a link with the classical notion of (smooth) solution. However,
this classical notion is deeply grounded on the concept of composition
of functions and, at the same time, it is often too narrow, as is
amply demonstrated e.g.~in the study of PDE in the presence of singularities.
In our opinion, it is at least not surprising that also the notion
of distributional solution did not lead to a satisfying theory of
nonlinear PDE (not even of singular ODE). We have hence a wild garden
of flourishing equation-dependent techniques and a zoo of counter-examples.
The well-known detaching between these techniques and numerical solutions
of PDE is another side of the same question.

One can say that this situation presents several analogies with the
classical compass-and-straightedge solution of geometrical problems,
or with the solution of polynomial equations by radicals. The distinction
between algebraic and irrational numbers and the advent of Galois
theory were essential steps for mathematics to start focusing on a
different concept of solution, frequently nearer to applied problems.
In the end, these classical problems stimulated more general notions
of geometrical transformation and numerical solution, which nowadays
have superseded their origins. The analogies are even greater when
observing that first steps toward a Galois theory of nonlinear PDE
are arising, see \cite{Ber96,Cas09,Mal02,Mal10}.

\medskip{}

Generalized smooth functions (GSF) are a possible formalization of
the original historical approach of the aforementioned classical authors.
We extend the field of real numbers into a natural non Archimedean
ring $\RC{\rho}$ and we consider the simplest notion of smooth function
on the extended ring of scalars $\rti$. To define a GSF $f:X\ra Y$,
$X\subseteq\rti^{n}$, $Y\subseteq\rti^{d}$, we simply require the
minimal logical conditions so that a net of ordinary smooth functions
$f_{\eps}\in\Coo(\Omega_{\eps},\R^{d})$, $\Omega_{\eps}\subseteq\R^{n}$,
defines a set-theoretical map $X\ra Y$ which is infinitely differentiable;
see below for the details. This freedom in the choice of domains and
codomains is a key property to prove that GSF are closed with respect
to composition. As a result, GSF share so many properties with ordinary
smooth functions that frequently we only have to formally generalize
classical proofs to the new context. This allows an easier approach
to this new theory of generalized functions.

It is important to note that the new framework is richer than the
classical one because of the possibility to express non-Archimedean
properties. So, e.g., two different infinitesimal standard deviations
in a Gaussian result in infinitely close Dirac-delta-like functionals
but, generally speaking, these two GSF could have different infinite
values at infinitesimal points $h\in\rti$. For this reason, Schwartz
distributions are embedded as GSF, but this embedding is not intrinsic
and it has to be chosen depending on the physical problem or on the
particular differential equation we aim to solve.

In the present work, we establish several inverse function theorems
for GSF. We prove both the classical local and also some global versions
of this theorem. It is remarkable to note that the local version is
formally very similar to the classical one, but with the sharp topology
instead of the standard Euclidean one. We also show the relations
between our results and the inverse function theorem for Colombeau
functions established by using the discontinuous calculus of \cite{AFJ05,AFJO12}.

The paper is self-contained in the sense that it contains all the
statements of results required for the proofs of the new inverse function
theorems. If proofs of preliminaries are omitted, we give references
to where they can be found.

\subsection{Basic notions}

\subsubsection*{The ring of generalized scalars}

In this work, $I$ denotes the interval $(0,1]\subseteq\R$ and we
will always use the variable $\eps$ for elements of $I$; we also
denote $\eps$-dependent nets $x\in\R^{I}$ simply by $(x_{\eps})$.
By $\N$ we denote the set of natural numbers, including zero.

We start by defining the non-Archimedean ring of scalars that extends
the real field $\R$. For all the proofs of results in this section,
see \cite{GKV,Gi-Ku-St15}.
\begin{definition}
\label{def:RCGN}Let $\rho=(\rho_{\eps})\in\R^{I}$ be a net such
that $\lim_{\eps\to0^{+}}\rho_{\eps}=0^{+}$, then 
\begin{enumerate}
\item $\mathcal{I}(\rho):=\left\{ (\rho_{\eps}^{-a})\mid a\in\R_{>0}\right\} $
is called the \emph{asymptotic gauge} generated by $\rho$. 
\item If $\mathcal{P}(\eps)$ is a property of $\eps\in I$, we use the
notation $\forall^{0}\eps:\,\mathcal{P}(\eps)$ to denote $\exists\eps_{0}\in I\,\forall\eps\in(0,\eps_{0}]:\,\mathcal{P}(\eps)$.
We can read $\forall^{0}\eps$ as \emph{for $\eps$ small}. 
\item We say that a net $(x_{\eps})\in\R^{I}$ \emph{is $\rho$-moderate},
and we write $(x_{\eps})\in\R_{\rho}$ if $\exists(J_{\eps})\in\mathcal{I}(\rho):\ x_{\eps}=O(J_{\eps})$
as $\eps\to0^{+}$. 
\item Let $(x_{\eps})$, $(y_{\eps})\in\R^{I}$, then we say that $(x_{\eps})\sim_{\rho}(y_{\eps})$
if $\forall(J_{\eps})\in\mathcal{I}(\rho):\ x_{\eps}=y_{\eps}+O(J_{\eps}^{-1})$
as $\eps\to0^{+}$. This is an equivalence relation on the ring $\R_{{\scriptscriptstyle \rho}}$
of moderate nets with respect to pointwise operations, and we can
hence define 
\[
\RC{\rho}:=\R_{{\scriptscriptstyle \rho}}/\sim_{\rho},
\]
which we call \emph{Robinson-Colombeau ring of generalized numbers},
\cite{Rob73,Col92}. We denote the equivalence class $x\in\rti$ simply
by $x=:[x_{\eps}]:=[(x_{\eps})]_{\sim}\in\rti$. 
\end{enumerate}
\end{definition}
In the following, $\rho$ will always denote a net as in Def.~\ref{def:RCGN},
and we will use the simpler notation $\Rtil$ for the case $\rho_{\eps}=\eps$.
The infinitesimal $\rho$ can be chosen depending on the class of
differential equations we need to solve for the generalized functions
we are going to introduce, see \cite{GiLu15}. For motivations concerning
the naturality of $\rti$, see \cite{Gi-Ku-St15}. We also use the
notation $\diff{\rho}:=[\rho_{\eps}]\in\rti$ and $\diff{\eps}:=[\eps]\in\RC{(\eps)}$.

We can also define an order relation on $\RC{\rho}$ by saying $[x_{\eps}]\le[y_{\eps}]$
if there exists $(z_{\eps})\in\R^{I}$ such that $(z_{\eps})\sim_{\rho}0$
(we then say that $(z_{\eps})$ is \emph{$\rho$-negligible}) and
$x_{\eps}\le y_{\eps}+z_{\eps}$ for $\eps$ small. Equivalently,
we have that $x\le y$ if and only if there exist representatives
$(x_{\eps})$, $(y_{\eps})$ of $x$, $y$ such that $x_{\eps}\le y_{\eps}$
for all $\eps$. Clearly, $\RC{\rho}$ is a partially ordered ring.
The usual real numbers $r\in\R$ are embedded in $\RC{\rho}$ considering
constant nets $[r]\in\RC{\rho}$.

Even if the order $\le$ is not total, we still have the possibility
to define the infimum $[x_{\eps}]\wedge[y_{\eps}]:=[\min(x_{\eps},y_{\eps})]$,
and analogously the supremum function $[x_{\eps}]\vee[y_{\eps}]:=\left[\max(x_{\eps},y_{\eps})\right]$
and the absolute value $|[x_{\eps}]|:=[|x_{\eps}|]\in\RC{\rho}$.
Our notations for intervals are: $[a,b]:=\{x\in\RC{\rho}\mid a\le x\le b\}$,
$[a,b]_{\R}:=[a,b]\cap\R$, and analogously for segments $[x,y]:=\left\{ x+r\cdot(y-x)\mid r\in[0,1]\right\} \subseteq\RC{\rho}^{n}$
and $[x,y]_{\R^{n}}=[x,y]\cap\R^{n}$. Finally, we write $x\approx y$
to denote that $|x-y|$ is an infinitesimal number, i.e.~$|x-y|\le r$
for all $r\in\R_{>0}$. This is equivalent to $\lim_{\eps\to0^{+}}|x_{\eps}-y_{\eps}|=0$
for all representatives $(x_{\eps})$, $(y_{\eps})$ of $x$, $y$.

\subsubsection*{Topologies on $\RC{\rho}^{n}$}

On the $\RC{\rho}$-module $\RC{\rho}^{n}$, we can consider the natural
extension of the Euclidean norm, i.e.~$|[x_{\eps}]|:=[|x_{\eps}|]\in\RC{\rho}$,
where $[x_{\eps}]\in\RC{\rho}^{n}$. Even if this generalized norm
takes values in $\RC{\rho}$, it shares several properties with usual
norms, like the triangular inequality or the property $|y\cdot x|=|y|\cdot|x|$.
It is therefore natural to consider on $\RC{\rho}^{n}$ topologies
generated by balls defined by this generalized norm and suitable notions
of being ``strictly less than a given radius'': 
\begin{definition}
\label{def:setOfRadii}Let $c\in\RC{\rho}^{n}$ and $x$, $y\in\RC{\rho}$,
then: 
\begin{enumerate}
\item We write $x<y$ if $\exists r\in\rti_{\ge0}:\ r\text{ is invertible, and }r\le y-x$ 
\item We write $x<_{\R}y$ if $\exists r\in\R_{>0}:\ r\le y-x$. 
\item $B_{r}(c):=\left\{ x\in\RC{\rho}^{n}\mid\left|x-c\right|<r\right\} $
for each $r\in\rti_{>0}$.
\item $B_{r}^{{\scriptscriptstyle \text{F}}}(c):=\left\{ x\in\RC{\rho}^{n}\mid\left|x-c\right|<_{\R}r\right\} $
for each $r\in\R_{>0}$. 
\item $\Eball_{r}(c):=\{x\in\R^{n}\mid|x-c|<r\}$, for each $r\in\R_{>0}$,
denotes an ordinary Euclidean ball in $\R^{n}$. 
\end{enumerate}
\end{definition}
\noindent The relations $<$, $<_{\R}$ have better topological properties
as compared to the usual strict order relation $a\le b$ and $a\ne b$
(that we will \emph{never} use) because both the sets of balls $\left\{ B_{r}(c)\mid r\in\rti_{>0},\ c\in\RC{\rho}^{n}\right\} $
and $\left\{ B_{r}^{{\scriptscriptstyle \text{F}}}(c)\mid r\in\R_{>0},\ c\in\RC{\rho}^{n}\right\} $
are bases for two topologies on $\RC{\rho}^{n}$. The former is called
\emph{sharp topology}, whereas the latter is called \emph{Fermat topology}.
We will call \emph{sharply open set} any open set in the sharp topology,
and \emph{large open set} any open set in the Fermat topology; clearly,
the latter is coarser than the former. The existence of infinitesimal
neighborhoods implies that the sharp topology induces the discrete
topology on $\R$. This is a necessary result when one has to deal
with continuous generalized functions which have infinite derivatives.
In fact, if $f'(x_{0})$ is infinite, only for $x\approx x_{0}$ we
can have $f(x)\approx f(x_{0})$.

The following result is useful to deal with positive and invertible
generalized numbers (cf.~\cite{GKOS,May08}). 
\begin{lemma}
\label{lem:mayer} Let $x\in\RC{\rho}$. Then the following are equivalent: 
\begin{enumerate}
\item \label{enu:positiveInvertible}$x$ is invertible and $x\ge0$, i.e.~$x>0$. 
\item \label{enu:strictlyPositive}For each representative $(x_{\eps})\in\R_{\rho}$
of $x$ we have $\forall^{0}\eps:\ x_{\eps}>0$. 
\item \label{enu:greater-i_epsTom}For each representative $(x_{\eps})\in\R_{\rho}$
of $x$ we have $\exists m\in\N\,\forall^{0}\eps:\ x_{\eps}>\rho_{\eps}^{m}$ 
\end{enumerate}
\end{lemma}

\subsubsection*{Internal and strongly internal sets}

A natural way to obtain sharply open, closed and bounded sets in $\RC{\rho}^{n}$
is by using a net $(A_{\eps})$ of subsets $A_{\eps}\subseteq\R^{n}$.
We have two ways of extending the membership relation $x_{\eps}\in A_{\eps}$
to generalized points $[x_{\eps}]\in\RC{\rho}$: 
\begin{definition}
\label{def:internalStronglyInternal}Let $(A_{\eps})$ be a net of
subsets of $\R^{n}$, then 
\begin{enumerate}
\item $[A_{\eps}]:=\left\{ [x_{\eps}]\in\RC{\rho}^{n}\mid\forall^{0}\eps:\,x_{\eps}\in A_{\eps}\right\} $
is called the \emph{internal set} generated by the net $(A_{\eps})$.
See \cite{ObVe08} for the introduction and an in-depth study of this
notion. 
\item Let $(x_{\eps})$ be a net of points of $\R^{n}$, then we say that
$x_{\eps}\in_{\eps}A_{\eps}$, and we read it as $(x_{\eps})$ \emph{strongly
belongs to $(A_{\eps})$}, if $\forall^{0}\eps:\ x_{\eps}\in A_{\eps}$
and if $(x'_{\eps})\sim_{\rho}(x_{\eps})$, then also $x'_{\eps}\in A_{\eps}$
for $\eps$ small. Moreover, we set $\sint{A_{\eps}}:=\left\{ [x_{\eps}]\in\RC{\rho}^{n}\mid x_{\eps}\in_{\eps}A_{\eps}\right\} $,
and we call it the \emph{strongly internal set} generated by the net
$(A_{\eps})$. 
\item Finally, we say that the internal set $K=[A_{\eps}]$ is \emph{sharply
bounded} if there exists $r\in\RC{\rho}_{>0}$ such that $K\subseteq B_{r}(0)$.
Analogously, a net $(A_{\eps})$ is \emph{sharply bounded} if there
exists $r\in\rti_{>0}$ such that $[A_{\eps}]\subseteq B_{r}(0)$. 
\end{enumerate}
\end{definition}
\noindent Therefore, $x\in[A_{\eps}]$ if there exists a representative
$(x_{\eps})$ of $x$ such that $x_{\eps}\in A_{\eps}$ for $\eps$
small, whereas this membership is independent from the chosen representative
in the case of strongly internal sets. Note explicitly that an internal
set generated by a constant net $A_{\eps}=A\subseteq\R^{n}$ is simply
denoted by $[A]$.

The following theorem shows that internal and strongly internal sets
have dual topological properties: 
\begin{theorem}
\noindent \label{thm:strongMembershipAndDistanceComplement}For $\eps\in I$,
let $A_{\eps}\subseteq\R^{n}$ and let $x_{\eps}\in\R^{n}$. Then
we have 
\begin{enumerate}
\item \label{enu:internalSetsDistance}$[x_{\eps}]\in[A_{\eps}]$ if and
only if $\forall q\in\R_{>0}\,\forall^{0}\eps:\ d(x_{\eps},A_{\eps})\le\rho_{\eps}^{q}$.
Therefore $[x_{\eps}]\in[A_{\eps}]$ if and only if $[d(x_{\eps},A_{\eps})]=0\in\RC{\rho}$. 
\item \label{enu:stronglyIntSetsDistance}$[x_{\eps}]\in\sint{A_{\eps}}$
if and only if $\exists q\in\R_{>0}\,\forall^{0}\eps:\ d(x_{\eps},A_{\eps}^{c})>\rho_{\eps}^{q}$,
where $A_{\eps}^{c}:=\R^{n}\setminus A_{\eps}$. Therefore, if $(d(x_{\eps},A_{\eps}^{c}))\in\R_{\rho}$,
then $[x_{\eps}]\in\sint{A_{\eps}}$ if and only if $[d(x_{\eps},A_{\eps}^{c})]>0$. 
\item \label{enu:internalAreClosed}$[A_{\eps}]$ is sharply closed and
$\sint{A_{\eps}}$ is sharply open. 
\item \label{enu:internalGeneratedByClosed}$[A_{\eps}]=\left[\text{\emph{cl}}\left(A_{\eps}\right)\right]$,
where $\text{\emph{cl}}\left(S\right)$ is the closure of $S\subseteq\R^{n}$.
On the other hand $\sint{A_{\eps}}=\sint{\text{\emph{int}\ensuremath{\left(A_{\eps}\right)}}}$,
where $\emph{int}\left(S\right)$ is the interior of $S\subseteq\R^{n}$. 
\end{enumerate}
\end{theorem}
We will also use the following:
\begin{lemma}
\noindent \label{Lem:fromInclusionTo-epsSmallIncl}Let $(\Omega_{\eps})$
be a net of subsets in $\R^{n}$ for all $\eps$, and $(B_{\eps})$
a sharply bounded net such that $[B_{\eps}]\subseteq\sint{\Omega_{\eps}}$,
then 
\[
\forall^{0}\eps:\ B_{\eps}\subseteq\Omega_{\eps}.
\]
\end{lemma}
Sharply bounded internal sets (which are always sharply closed by
Thm.~\ref{thm:strongMembershipAndDistanceComplement}~\ref{enu:internalAreClosed})
serve as compact sets for our generalized functions. For a deeper
study of this type of sets in the case $\rho=(\eps)$ see \cite{ObVe08,GiKu16};
in the same particular setting, see \cite{GKV} and references therein
for (strongly) internal sets.

\subsubsection*{Generalized smooth functions}

For the ideas presented in this section, see also e.g.~\cite{GKV,Gi-Ku-St15}.

Using the ring $\rti$, it is easy to consider a Gaussian with an
infinitesimal standard deviation. If we denote this probability density
by $f(x,\sigma)$, and if we set $\sigma=[\sigma_{\eps}]\in\RC{\rho}_{>0}$,
where $\sigma\approx0$, we obtain the net of smooth functions $(f(-,\sigma_{\eps}))_{\eps\in I}$.
This is the basic idea we develop in the following 
\begin{definition}
\label{def:netDefMap}Let $X\subseteq\RC{\rho}^{n}$ and $Y\subseteq\RC{\rho}^{d}$
be arbitrary subsets of generalized points. Then we say that 
\[
f:X\longrightarrow Y\text{ is a \emph{generalized smooth function}}
\]
if there exists a net of functions $f_{\eps}\in\cinfty(\Omega_{\eps},\R^{d})$
defining $f$ in the sense that $X\subseteq\langle\Omega_{\eps}\rangle$,
$f([x_{\eps}])=[f_{\eps}(x_{\eps})]\in Y$ and $(\partial^{\alpha}f_{\eps}(x_{\eps}))\in\R_{{\scriptscriptstyle \rho}}^{d}$
for all $x=[x_{\eps}]\in X$ and all $\alpha\in\N^{n}$. The space
of GSF from $X$ to $Y$ is denoted by $\gsf(X,Y)$. 
\end{definition}
Let us note explicitly that this definition states minimal logical
conditions to obtain a set-theoretical map from $X$ into $Y$ and
defined by a net of smooth functions. In particular, the following
Thm.~\ref{thm:propGSF} states that the equality $f([x_{\eps}])=[f_{\eps}(x_{\eps})]$
is meaningful, i.e.~that we have independence from the representatives
for all derivatives $[x_{\eps}]\in X\mapsto[\partial^{\alpha}f_{\eps}(x_{\eps})]\in\RC{\rho}^{d}$,
$\alpha\in\N^{n}$. 
\begin{theorem}
\label{thm:propGSF}Let $X\subseteq\RC{\rho}^{n}$ and $Y\subseteq\RC{\rho}^{d}$
be arbitrary subsets of generalized points. Let $f_{\eps}\in\cinfty(\Omega_{\eps},\R^{d})$
be a net of smooth functions that defines a generalized smooth map
of the type $X\longrightarrow Y$, then 
\begin{enumerate}
\item $\forall\alpha\in\N^{n}\,\forall(x_{\eps}),(x'_{\eps})\in\R_{\rho}^{n}:\ [x_{\eps}]=[x'_{\eps}]\in X\ \Rightarrow\ (\partial^{\alpha}u_{\eps}(x_{\eps}))\sim_{\rho}(\partial^{\alpha}u_{\eps}(x'_{\eps}))$. 
\item \label{enu:modOnEpsDepBall}$\forall[x_{\eps}]\in X\,\forall\alpha\in\N^{n}\,\exists q\in\R_{>0}\,\forall^{0}\eps:\ \sup_{y\in\Eball_{\eps^{q}}(x_{\eps})}\left|\partial^{\alpha}u_{\eps}(y)\right|\le\eps^{-q}$. 
\item \label{enu:locLipSharp}For all $\alpha\in\N^{n}$, the GSF $g:[x_{\eps}]\in X\mapsto[\partial^{\alpha}f_{\eps}(x_{\eps})]\in\Rtil^{d}$
is locally Lipschitz in the sharp topology, i.e.~each $x\in X$ possesses
a sharp neighborhood $U$ such that $|g(x)-g(y)|\le L|x-y|$ for all
$x$, $y\in U$ and some $L\in\RC{\rho}$. 
\item \label{enu:GSF-cont}Each $f\in\gsf(X,Y)$ is continuous with respect
to the sharp topologies induced on $X$, $Y$. 
\item \label{enu:suffCondFermatCont}Assume that the GSF $f$ is locally
Lipschitz in the Fermat topology and that its Lipschitz constants
are always finite: $L\in\R$. Then $f$ is continuous in the Fermat
topology. 
\item \label{enu:globallyDefNet}$f:X\longrightarrow Y$ is a GSF if and
only if there exists a net $v_{\eps}\in\cinfty(\R^{n},\R^{d})$ defining
a generalized smooth map of type $X\longrightarrow Y$ such that $f=[v_{\eps}(-)]|_{X}$. 
\item \label{enu:category}Subsets $S\subseteq\RC{\rho}^{s}$ with the trace
of the sharp topology, and generalized smooth maps as arrows form
a subcategory of the category of topological spaces. We will call
this category $\gsf$, the \emph{category of GSF}. 
\end{enumerate}
\end{theorem}
The differential calculus for GSF can be introduced showing existence
and uniqueness of another GSF serving as incremental ratio. For its
statement, if $\mathcal{P}(h)$ is a property of $h\in\rti$, then
we write $\forall^{\text{s}}h:\ \mathcal{P}(h)$ to denote $\exists r\in\rti_{>0}\,\forall h\in B_{r}(0):\ \mathcal{P}(h)$
and $\forall^{{\scriptscriptstyle \text{F}}}h:\ \mathcal{P}(h)$ for
$\exists r\in\R_{>0}\,\forall h\in B_{r}^{{\scriptscriptstyle \text{F}}}(c):\ \mathcal{P}(h)$. 
\begin{theorem}
\noindent \label{thm:FR-forGSF} Let $U\subseteq\RC{\rho}^{n}$ be
a sharply open set, let $v=[v_{\eps}]\in\RC{\rho}^{n}$, and let $f\in\gsf(U,\RC{\rho})$
be a generalized smooth map generated by the net of smooth functions
$f_{\eps}\in\cinfty(\Omega_{\eps},\R)$. Then 
\begin{enumerate}
\item \label{enu:existenceRatio}There exists a sharp neighborhood $T$
of $U\times\{0\}$ and a generalized smooth map $r\in\gsf(T,\RC{\rho})$,
called the \emph{generalized incremental ratio} of $f$ \emph{along}
$v$, such that 
\[
\forall x\in U\,\forall^{\text{s}}h:\ f(x+hv)=f(x)+h\cdot r(x,h).
\]
\item \label{enu:uniquenessRatio}If $\bar{r}\in\gsf(S,\RC{\rho})$ is another
generalized incremental ratio of $f$ along $v$ defined on a sharp
neighborhood $S$ of $U\times\{0\}$, then 
\[
\forall x\in U\,\forall^{\text{s}}h:\ r(x,h)=\bar{r}(x,h).
\]
\item \label{enu:defDer}We have $r(x,0)=\left[\frac{\partial f_{\eps}}{\partial v_{\eps}}(x_{\eps})\right]$
for every $x\in U$ and we can thus define $\frac{\partial f}{\partial v}(x):=r(x,0)$,
so that $\frac{\partial f}{\partial v}\in\gsf(U,\RC{\rho})$. 
\end{enumerate}
\noindent If $U$ is a large open set, then an analogous statement
holds replacing $\forall^{\text{s}}h$ by $\forall^{{\scriptscriptstyle \text{F}}}h$
and sharp neighborhoods by large neighborhoods. 
\end{theorem}
Note that this result permits to consider the partial derivative of
$f$ with respect to an arbitrary generalized vector $v\in\RC{\rho}^{n}$
which can be, e.g., infinitesimal or infinite.

Using this result we obtain the usual rules of differential calculus,
including the chain rule. Finally, we note that for each $x\in U$,
the map $Df(x).v:=\frac{\partial f}{\partial v}(x)\in\RC{\rho}^{d}$
is $\RC{\rho}$-linear in $v\in\RC{\rho}^{n}$. The set of all the
$\rti$-linear maps $\rti^{n}\ra\rti^{d}$ will be denoted by $L(\rti^{n},\rti^{d})$.
For $A=[A_{\eps}(-)]\in L(\rti^{n},\rti^{d})$, we set $\no{A}:=[\no{A_{\eps}}]$,
the generalized number defined by the operator norms of the matrices
$A_{\eps}\in L(\R^{n},\R^{d})$.

\subsubsection*{Embedding of Schwartz distributions and Colombeau functions}

We finally recall two results that give a certain flexibility in constructing
embeddings of Schwartz distributions. Note that both the infinitesimal
$\rho$ and the embedding of Schwartz distributions have to be chosen
depending on the problem we aim to solve. A trivial example in this
direction is the ODE $y'=y/\diff{\eps}$, which cannot be solved for
$\rho=(\eps)$, but it has a solution for $\rho=(e^{-1/\eps})$. As
another simple example, if we need the property $H(0)=1/2$, where
$H$ is the Heaviside function, then we have to choose the embedding
of distributions accordingly. This corresponds to the philosophy followed
in \cite{Hair14}. See also \cite{GiLu15} for further details.\\
 If $\phi\in\mathcal{D}(\R^{n})$, $r\in\R_{>0}$ and $x\in\R^{n}$,
we use the notations $r\odot\phi$ for the function $x\in\R^{n}\mapsto\frac{1}{r^{n}}\cdot\phi\left(\frac{x}{r}\right)\in\R$
and $x\oplus\phi$ for the function $y\in\R^{n}\mapsto\phi(y-x)\in\R$.
These notations permit to highlight that $\odot$ is a free action
of the multiplicative group $(\R_{>0},\cdot,1)$ on $\mathcal{D}(\R^{n})$
and $\oplus$ is a free action of the additive group $(\R_{>0},+,0)$
on $\mathcal{D}(\R^{n})$. We also have the distributive property
$r\odot(x\oplus\phi)=rx\oplus r\odot\phi$. 
\begin{lemma}
\label{lem:strictDeltaNet}Let $b\in\R_{{\scriptscriptstyle \rho}}$
be a net such that $\lim_{\eps\to0^{+}}b_{\eps}=+\infty$. Let $d\in(0,1)$.
There exists a net $\left(\psi_{\eps}\right)_{\eps\in I}$ of $\mathcal{D}(\R^{n})$
with the properties: 
\begin{enumerate}
\item \label{enu:suppStrictDeltaNet}$supp(\psi_{\eps})\subseteq B_{1}(0)$
for all $\eps\in I$. 
\item \label{enu:intOneStrictDeltaNet}$\int\psi_{\eps}=1$ for all $\eps\in I$. 
\item \label{enu:moderateStrictDeltaNet}$\forall\alpha\in\N^{n}\,\exists p\in\N:\ \sup_{x\in\R^{n}}\left|\partial^{\alpha}\psi_{\eps}(x)\right|=O(b_{\eps}^{p})$
as $\eps\to0^{+}$. 
\item \label{enu:momentsStrictDeltaNet}$\forall j\in\N\,\forall^{0}\eps:\ 1\le|\alpha|\le j\Rightarrow\int x^{\alpha}\cdot\psi_{\eps}(x)\diff{x}=0$. 
\item \label{enu:smallNegPartStrictDeltaNet}$\forall\eta\in\R_{>0}\,\forall^{0}\eps:\ \int\left|\psi_{\eps}\right|\le1+\eta$. 
\item \label{enu:int1Dim}If $n=1$, then the net $(\psi_{\eps})_{\eps\in I}$
can be chosen so that $\int_{-\infty}^{0}\psi_{\eps}=d$. 
\end{enumerate}
\noindent If $\psi_{\eps}$ satisfies \ref{enu:suppStrictDeltaNet}
\textendash{} \ref{enu:int1Dim} then in particular $\psi_{\eps}^{b}:=b_{\eps}^{-1}\odot\psi_{\eps}$
satisfies \ref{enu:intOneStrictDeltaNet} - \ref{enu:smallNegPartStrictDeltaNet}. 
\end{lemma}
\noindent Concerning embeddings of Schwartz distributions, we have
the following result, where $\csp{\Omega}:=\{[x_{\eps}]\in[\Omega]\mid\exists K\Subset\Omega\,\forall^{0}\eps:\ x_{\eps}\in K\}$
is called the set of \emph{compactly supported points in }$\Omega\subseteq\R^{n}$. 
\begin{theorem}
\label{thm:embeddingD'}Under the assumptions of Lemma \ref{lem:strictDeltaNet},
let $\Omega\subseteq\R^{n}$ be an open set and let $(\psi_{\eps}^{b})$
be the net defined in Lemma \ref{lem:strictDeltaNet}. Then the mapping
\[
\iota_{\Omega}^{b}:T\in\mathcal{E}'(\Omega)\mapsto\left[\left(T\ast\psi_{\eps}^{b}\right)(-)\right]\in\gsf(\csp{\Omega},\rti)
\]
uniquely extends to a sheaf morphism of real vector spaces 
\[
\iota^{b}:\mathcal{D}'\ra\gsf(\csp{(-)},\rti),
\]
and satisfies the following properties: 
\begin{enumerate}
\item If $b\ge\diff{\rho}^{-a}$ for some $a\in\R_{>0}$, then $\iota^{b}|_{\Coo(-)}:\Coo(-)\ra\gsf(\csp{(-)},\RC{\rho})$
is a sheaf morphism of algebras.
\item If $T\in\mathcal{E}'(\Omega)$ then $\text{\text{\emph{supp}}}(T)=\text{\emph{\text{supp}}}(\iota_{\Omega}^{b}(T))$.
\item $\lim_{\eps\to0^{+}}\int_{\Omega}\iota_{\Omega}^{b}(T)_{\eps}\cdot\phi=\langle T,\phi\rangle$
for all $\phi\in\mathcal{D}(\Omega)$ and all $T\in\mathcal{D}'(\Omega)$.
\item $\iota^{b}$ commutes with partial derivatives, i.e.~$\partial^{\alpha}\left(\iota_{\Omega}^{b}(T)\right)=\iota_{\Omega}^{b}\left(\partial^{\alpha}T\right)$
for each $T\in\mathcal{D}'(\Omega)$ and $\alpha\in\N$. 
\end{enumerate}
\end{theorem}
Concerning the embedding of Colombeau generalized functions, we recall
that the special Colombeau algebra on $\Om$ is defined as the quotient
$\gs(\Om):=\esm(\Om)/\ns(\Om)$ of \emph{moderate nets} over \emph{negligible
nets}, where the former is 
\begin{multline*}
\esm(\Om):=\{(u_{\eps})\in\cinfty(\Omega)^{I}\mid\forall K\comp\Om\,\forall\alpha\in\N^{n}\,\exists N\in\N:\sup_{x\in K}|\partial^{\alpha}u_{\eps}(x)|=O(\eps^{-N})\}
\end{multline*}
and the latter is 
\begin{multline*}
\ns(\Om):=\{(u_{\eps})\in\cinfty(\Omega)^{I}\mid\forall K\comp\Om\,\forall\alpha\in\N^{n}\,\forall m\in\N:\sup_{x\in K}|\partial^{\alpha}u_{\eps}(x)|=O(\eps^{m})\}.
\end{multline*}
Using $\rho=(\eps)$, we have the following compatibility result: 
\begin{theorem}
\label{thm:inclusionCGF}A Colombeau generalized function $u=(u_{\eps})+\ns(\Om)^{d}\in\gs(\Omega)^{d}$
defines a generalized smooth map $u:[x_{\eps}]\in\csp{\Omega}\longrightarrow[u_{\eps}(x_{\eps})]\in\Rtil^{d}$
which is locally Lipschitz on the same neighborhood of the Fermat
topology for all derivatives. This assignment provides a bijection
of $\gs(\Omega)^{d}$ onto $\gsf(\csp{\Omega},\rti^{d})$ for every
open set $\Omega\subseteq\R^{n}$. 
\end{theorem}
For GSF, suitable generalizations of many classical theorems of differential
and integral calculus hold: intermediate value theorem, mean value
theorems, Taylor formulas in different forms, a sheaf property for
the Fermat topology, and the extreme value theorem on internal sharply
bounded sets (see \cite{Gi-Ku-St15}). The latter are called \emph{functionally
compact} subsets of $\rti^{n}$ and serve as compact sets for GSF.
A theory of compactly supported GSF has been developed in \cite{GiKu16},
and it closely resembles the classical theory of LF-spaces of compactly
supported smooth functions. It results that for suitable functionally
compact subsets, the corresponding space of compactly supported GSF
contains extensions of all Colombeau generalized functions, and hence
also of all Schwartz distributions. Finally, in these spaces it is
possible to prove the Banach fixed point theorem and a corresponding
Picard-Lindelöf theorem, see \cite{GiLu16}.

\section{Local inverse function theorems}

As in the case of classical smooth functions, any infinitesimal criterion
for the invertibility of generalized smooth functions will rely on
the invertibility of the corresponding differential. We therefore
note the following analogue of \cite[Lemma 1.2.41]{GKOS} (whose proof
transfers literally to the present situation): 
\begin{lemma}
\label{lem:invertmatrix} Let $A\in\rti^{n\times n}$ be a square
matrix. The following are equivalent: 
\begin{enumerate}
\item $A$ is nondegenerate, i.e., $\xi\in\rti^{n}$, $\xi^{t}A\eta=0$
$\forall\eta\in\rti^{n}$ implies $\xi=0$. 
\item $A:\rti^{n}\to\rti^{n}$ is injective. 
\item $A:\rti^{n}\to\rti^{n}$ is surjective. 
\item $\det(A)$ is invertible. 
\end{enumerate}
\end{lemma}
\begin{theorem}
\label{thm:localIFTSharp}Let $X\sse\rti^{n}$, let $f\in\gsf(X,\rti^{n})$
and suppose that for some $x_{0}$ in the sharp interior of $X$,
$Df(x_{0})$ is invertible in $L(\rti^{n},\rti^{n})$. Then there
exists a sharp neighborhood $U\sse X$ of $x_{0}$ and a sharp neighborhood
$V$ of $f(x_{0})$ such that $f:U\to V$ is invertible and $f^{-1}\in\gsf(V,U)$. 
\end{theorem}
\begin{proof}
Thm.~\ref{thm:propGSF}.\ref{enu:globallyDefNet} entails that $f$
can be defined by a globally defined net $f_{\eps}\in\cinfty(\R^{n},\R^{n})$.
Hadamard's inequality (cf.~\cite[Prop. 3.43]{EDiss}) implies $\no{Df(x_{0})^{-1}}\ge\sqrt[n]{\frac{1}{C}\left|\det\left(Df(x_{0})^{-1}\right)\right|}$,
where $C\in\R_{>0}$ is a universal constant that only depends on
the dimension $n$. Thus, by Lemma \ref{lem:invertmatrix} and Lemma
\ref{lem:mayer}, $\det{Df(x_{0})}$ and consequently also $a:=\no{Df(x_{0})^{-1}}$
is invertible. Next, pick positive invertible numbers $b$, $r\in\rti$
such that $ab<1$, $B_{2r}(x_{0})\sse X$ and 
\[
\no{Df(x_{0})-Df(x)}<b
\]
for all $x\in B_{2r}(x_{0})$. Such a choice of $r$ is possible since
every derivative of $f$ is continuous with respect to the sharp topology
(see Thm.~\ref{thm:propGSF}.\ref{enu:GSF-cont} and Thm.~\ref{thm:FR-forGSF}.\ref{enu:defDer}).
Pick representatives $(a_{\eps})$, $(b_{\eps})$ and $(r_{\eps})$
of $a$, $b$ and $r$ such that for all $\eps\in I$ we have $b_{\eps}>0$,
$a_{\eps}b_{\eps}<1$, and $r_{\eps}>0$. Let $(x_{0\eps})$ be a
representative of $x_{0}$. Since $[B_{r_{\eps}}(x_{0\eps})]\subseteq B_{2r}(x_{0})$,
by Lemma \ref{Lem:fromInclusionTo-epsSmallIncl} we can also assume
that $B_{r_{\eps}}(x_{0\eps})\sse\Omega_{\eps}$, and $\no{Df_{\eps}(x_{0\eps})-Df_{\eps}(x)}<b_{\eps}$
for all $x\in U_{\eps}:=B_{r_{\eps}}(x_{0\eps})$. Now let $c_{\eps}:=\frac{a_{\eps}}{1-a_{\eps}b_{\eps}}$.
Then $c:=[c_{\eps}]>0$ and by \cite[Th. 6.4]{EG:13} we obtain for
each $\eps\in I$: 
\begin{enumerate}[label=(\alph*)]
\item \label{enu:x_0}For all $x\in U_{\eps}:=B_{r_{\eps}}(x_{0\eps})$,
$Df_{\eps}(x)$ is invertible and $\no{Df_{\eps}(x)^{-1}}\le c_{\eps}$. 
\item $V_{\eps}:=f_{\eps}(B_{r_{\eps}}(x_{0\eps}))$ is open in $\R^{n}$. 
\item $f_{\eps}|_{U_{\eps}}:U_{\eps}\ra V_{\eps}$ is a diffeomorphism,
and 
\item \label{enu:y_0}setting $y_{0\eps}:=f_{\eps}(x_{0\eps})$, we have
$B_{r_{\eps}/c_{\eps}}(y_{0\eps})\sse f_{\eps}(B_{r_{\eps}}(x_{0\eps}))$. 
\end{enumerate}
The sets $U:=\sint{U_{\eps}}=B_{r}(x_{0})\subseteq X$ and $V:=\sint{V_{\eps}}$
are sharp neighborhoods of $x_{0}$ and $f(x_{0})$, respectively,
by \ref{enu:y_0}, and so it remains to prove that $[f_{\eps}|_{U_{\eps}}^{-1}(-)]\in\gsf(V,U)$.

We first note that by \ref{enu:x_0}, $\no{Df_{\eps}(x)^{-1}}\le c_{\eps}$
for all $x\in B_{r_{\eps}}(x_{0\eps})$, which by Hadamard's inequality
implies 
\begin{equation}
|\det(Df_{\eps}(x))|\ge\frac{1}{C\cdot c_{\eps}^{n}}\qquad(x\in B_{r_{\eps}}(x_{0\eps})).\label{detlow}
\end{equation}
Now for $[y_{\eps}]\in V$ and $1\le i,j\le n$ we have (see e.g.~\cite[(3.15)]{EDiss})
\begin{equation}
\partial_{j}(f_{\eps}^{-1})^{i}(y_{\eps})=\frac{1}{\det(Df_{\eps}(f_{\eps}^{-1}(y_{\eps})))}\cdot P_{ij}((\partial_{s}f_{\eps}^{r}(f_{\eps}^{-1}(y_{\eps})))_{r,s}),\label{derest}
\end{equation}
where $P_{ij}$ is a polynomial in the entries of the matrix in its
argument. Since $[f_{\eps}^{-1}(y_{\eps})]\in U\subseteq X$, it follows
from \eqref{detlow} and the fact that $f|_{U}\in\gsf(U,\rti^{n})$
that 
\[
(\partial_{j}(f_{\eps}^{-1})^{i}(y_{\eps}))\in\R_{{\scriptscriptstyle \rho}}^{n}.
\]
Higher order derivatives can be treated analogously, thereby establishing
that every derivative of $g_{\eps}:=f_{\eps}|_{U_{\eps}}^{-1}$ is
moderate. To prove the claim, it remains to show that $[g_{\eps}(y_{\eps})]\in U=\sint{U_{\eps}}$
for all $[y_{\eps}]\in V=\sint{V_{\eps}}$. Since $g_{\eps}:V_{\eps}\ra U_{\eps}$,
we only prove that if $(x_{\eps})\sim_{\rho}(g_{\eps}(y_{\eps}))$,
then also $x_{\eps}\in U_{\eps}$ for $\eps$ small. We can set $y'_{\eps}:=f_{\eps}(x{}_{\eps})$
because $f_{\eps}$ is defined on the entire $\R^{n}$. By the mean
value theorem applied to $f_{\eps}$ and the moderateness of $f'$,
we get 
\[
\no{y'_{\eps}-y_{\eps}}=\no{f_{\eps}(x{}_{\eps})-f_{\eps}(g_{\eps}(y_{\eps}))}\le\rho_{\eps}^{N}\cdot\no{x_{\eps}-g_{\eps}(y_{\eps})}.
\]
Therefore $(y'_{\eps})\sim_{\rho}(y_{\eps})$ and hence $y'_{\eps}\in V_{\eps}$
and $g_{\eps}(y'_{\eps})=x_{\eps}\in U_{\eps}$ for $\eps$ small.
\hspace*{\fill}\qed 
\end{proof}

From Thm.~\ref{thm:propGSF}.\ref{enu:GSF-cont}, we know that any
generalized smooth function is sharply continuous. Thus we obtain: 
\begin{corollary}
\label{cor:localDiff}Let $X\sse\rti^{n}$ be a sharply open set,
and let $f\in\gsf(X,\rti^{n})$ be such that $Df(x)$ is invertible
for each $x\in X$. Then $f$ is a local homeomorphism with respect
to the sharp topology. In particular, it is an open map. 
\end{corollary}
Any such map $f$ will therefore be called a local generalized diffeomorphism.
If $f\in\gsf(X,Y)$ possesses an inverse in $\gsf(Y,X)$, then it
is called a global generalized diffeomorphism.

Following the same idea we used in the proof of Thm.~\ref{thm:localIFTSharp},
we can prove a sufficient condition to have a local generalized diffeomorphism
which is defined in a large neighborhood of $x_{0}$: 
\begin{theorem}
\label{thm:localIFTFermat}Let $X\sse\rti^{n}$, let $f\in\gsf(X,\rti^{n})$
and suppose that for some $x_{0}$ in the Fermat interior of $X$,
$Df(x_{0})$ is invertible in $L(\rti^{n},\rti^{n})$. Assume that
$\no{Df(x_{0})^{-1}}$ is finite, i.e.~$\no{Df(x_{0})^{-1}}\le k$
for some $k\in\R_{>0}$, and $Df$ is Fermat continuous. Then there
exists a large neighborhood $U\sse X$ of $x_{0}$ and a large neighborhood
$V$ of $f(x_{0})$ such that $f:U\to V$ is invertible and $f^{-1}\in\gsf(V,U)$.
\end{theorem}
\begin{proof}
We proceed as above, but now we have $r_{\eps}=r\in\R_{>0}$, $b_{\eps}=b\in\R_{>0}$
because of our assumptions. Setting $c_{\eps}:=\frac{a_{\eps}}{1-a_{\eps}b}$,
we have that $c:=[c_{\eps}]\in\rti_{>0}$ is finite. Therefore, there
exists $s\in\R_{>0}$ such that $s<\frac{r}{c}$. We can continue
as above, noting that now $B_{r}^{{\scriptscriptstyle \text{F}}}(x_{0})\subseteq U=B_{r}(x_{0})\subseteq X$
and $B_{s}^{{\scriptscriptstyle \text{F}}}(y_{0})\subseteq B_{r/c}(y_{0})\subseteq V$
are large neighborhoods of $x_{0}$ and $f(x_{0})$ respectively.
\hspace*{\fill}\qed 
\end{proof}

\begin{example}
\label{exa:examples}\ 
\begin{enumerate}
\item Thm.~\ref{thm:embeddingD'}, for $n=1$, shows that $\delta(x)=\left[b_{\eps}\psi_{\eps}\left(b_{\eps}x\right)\right]$
is, up to sheaf isomorphism, the Dirac delta. This also shows directly
that $\delta\in\gsf({\rti,\rti)}$. We can take the net $(\psi_{\eps})$
so that $\psi_{\eps}(0)=1$ for all $\eps$. In this way, $H'(0)=\delta(0)=b$
is an infinite number. We can thus apply the local inverse function
theorem \ref{thm:localIFTSharp} to the Heaviside function $H$ obtaining
that $H$ is a generalized diffeomorphism in an infinitesimal neighborhood
of $0$. This neighborhood cannot be finite because $H'(r)=0$ for
all $r\in\R_{\ne0}$. 
\item By the intermediate value theorem for GSF (see \cite[Cor.~42]{Gi-Ku-St15}),
in the interval $[0,1/2]$ the Dirac delta takes any value in $[0,\delta(0)]$.
So, let $k\in[0,1/2]$ such that $\delta(k)=1$. Then by the mean
value theorem for GSF (see \cite[Thm.~43]{Gi-Ku-St15}) $\delta(\delta(1))-\delta(\delta(k))=\delta(0)-\delta(1)=b-0=(\delta\circ\delta)'(c)\cdot(1-k)$
for some $c\in[k,1]$. Therefore $(\delta\circ\delta)'(c)=\frac{b}{1-k}\in\rti_{>0}$,
and around $c$ the composition $\delta\circ\delta$ is invertible.
Note that $(\delta\circ\delta)(r)=b$ for all $r\in\R_{\ne0}$, and
$(\delta\circ\delta)(h)=0$ for all $h\in\rti$ such that $\delta(h)$
is not infinitesimal. 
\end{enumerate}
Now, let $r\in\rti_{>0}$ be an infinitesimal generalized number,
i.e.~$r\approx0$. 
\begin{enumerate}[resume]
\item \label{enu:rx}Let $f(x):=r\cdot x$ for $x\in\csp{\R}$. Then $f'(x_{0})=r\approx0$
and Thm.~\ref{thm:localIFTSharp} yields $f^{-1}:y\in B_{s}(rx_{0})\mapsto y/r\in\csp{\R}$
for some $s\in\rti_{>0}$. But $y/r$ is finite only if $y$ is infinitesimal,
so that $s\approx0$. This shows that the assumption in Thm.~\ref{thm:localIFTFermat}
on $\no{Df(x_{0})^{-1}}$ being finite is necessary. 
\item Let $f(x):=\sin\frac{x}{r}$. We have $f\in\gsf(\rti,\rti)$ and $f'(x)=\frac{1}{r}\cos\frac{x}{r}$,
which is always an infinite number e.g.~if $\exists\lim_{\eps\to0^{+}}x_{\eps}\ne r(2k+1)\frac{\pi}{2}\approx0$,
$k\in\Z$. By Thm.~\ref{thm:localIFTSharp}, we know that $f$ is
invertible e.g.~around $x=0$. It is easy to recognize that $f$
is injective in the infinitesimal interval $\left(-\frac{\pi}{2}r,+\frac{\pi}{2}r\right)$.
In \cite[Exa.~3.9]{EDiss}, it is proved that $f$ is not injective
in any large neighborhood of $x=0$. Therefore, $\left(f|_{\left(-\frac{\pi}{2}r,+\frac{\pi}{2}r\right)}\right)^{-1}$
is a GSF that cannot be extended to a Colombeau generalized function. 
\item Similarly, $f(x):=r\sin x$, $x\in\rti$, has an inverse function
which cannot be extended outside the infinitesimal neighborhood $(-r,r)$. 
\item Thm.~\ref{thm:localIFTSharp} cannot be applied to $f(x):=x^{3}$
at $x_{0}=0$. However, if we restrict to $x\in(-\infty,-r)\cup(r,+\infty)$,
then the inverse function $f^{-1}(y)=y^{1/3}$ is defined in $y\in(-\infty,-r^{3})\cup(r^{3},+\infty)$
and has infinite derivative at each infinitesimal point in its domain. 
\end{enumerate}
\end{example}
In \cite{AFJ05}, Aragona, Fernandez and Juriaans introduced a differential
calculus on spaces of Colombeau generalized points based on a specific
form of convergence of difference quotients. Moreover, in \cite{AFJO12},
an inverse function theorem for Colombeau generalized functions in
this calculus was established. In the one-dimensional case it was
shown in \cite{GKV} that any GSF is differentiable in the sense of
\cite{AFJ05,AFJO12}, with the same derivative. Below we will show
that this compatibility is in fact true in arbitrary dimensions and
that Theorem \ref{thm:localIFTSharp} implies the corresponding result
from \cite{AFJO12}. In the remaining part of the present section,
we therefore restrict our attention to the case $\rho_{\eps}=\eps$, the gauge that
is used in standard Colombeau theory (as well as in \cite{AFJ05,AFJO12}),
and hence $\rti=\Rtil$ and $\csp{\Omega}=\otilc$.

First, we recall the definition from \cite{AFJ05}: 
\begin{definition}
A map $f$ from some sharply open subset $U$ of $\Rtil^{n}$ to $\Rtil^{m}$
is called differentiable in $x_{0}\in U$ in the sense of \cite{AFJ05}
with derivative $A\in L(\Rtil^{n},\Rtil^{m})$ if 
\begin{equation}
\lim_{x\to x_{0}}\frac{\no{f(x)-f(x_{0})-A(x-x_{0})}_{e}}{\no{x-x_{0}}_{e}}=0,\label{afjdiff}
\end{equation}
where 
\begin{align*}
 & v:(x_{\eps})\in\R_{(\eps)}^{n}\mapsto\sup\{b\in\R\mid\no{x_{\eps}}=O(\eps^{b})\}\in(-\infty,\infty]\\
 & \no{-}_{e}:x\in\Rtil^{n}\mapsto\exp(-v(x))\in[0,\infty).
\end{align*}
\end{definition}
The following result shows compatibility of this notion with the derivative
in the sense of GSF. 
\begin{lemma}
\label{afjcomp} Let $U$ be sharply open in $\Rtil^{n}$, let $x_{0}\in U$
and suppose that $f\in\gsfo(U,\Rtil^{m})$. Then $f$ is differentiable
in the sense of \cite{AFJ05} in $x_{0}$ with derivative $Df(x_{0})$. 
\end{lemma}
\begin{proof}
Without loss of generality we may suppose that $m=1$. Let $f$ be
defined by the net $f_{\eps}\in\cinfty(\R^{n},\R)$ for all $\eps$.
Since $(D^{2}f_{\eps}(x_{\eps}))$ is moderate, it follows from Thm.~\ref{thm:propGSF}.\ref{enu:modOnEpsDepBall}
that there exists some $q>0$ such that $\sup_{y\in B_{\eps^{q}}^{{\scriptscriptstyle \text{E}}}(x_{\eps})}\no{D^{2}f_{\eps}(y)}\le\eps^{-q}$
for $\eps$ small. Then by Taylor's theorem we have 
\begin{align*}
f_{\eps}(x_{\eps})- & f_{\eps}(x_{0\eps})-Df_{\eps}(x_{0\eps})(x_{\eps}-x_{0\eps})=\\
 & =\sum_{|\alpha|=2}\frac{|\alpha|}{\alpha!}\int_{0}^{1}(1-t)^{|\alpha|-1}\partial^{\alpha}f_{\eps}(x_{0\eps}+t(x_{\eps}-x_{0\eps}))\,\diff{t}\cdot(x_{\eps}-x_{0\eps})^{\alpha}.
\end{align*}
For $[x_{\eps}]\in B_{\text{d}\eps^{q}}(x_{0})$ this implies that
\[
\no{f(x)-f(x_{0})-Df(x_{0})(x-x_{0})}_{e}\le e^{q}\no{x-x_{0}}_{e}^{2},
\]
thereby establishing \eqref{afjdiff} with $A=Df(x_{0})$, as claimed.
\hspace*{\fill}\qed 
\end{proof}

It follows that any $f\in\gsfo(U,\Rtil^{m})$ is in fact even infinitely
often differentiable in the sense of \cite{AFJ05}.

Based on these observations we may now give an alternative proof for
\cite[Th. 3]{AFJO12}: 
\begin{theorem}
\label{locinvCol}Let $\Omega\subseteq\R^{n}$ be open, $f\in\gs(\Omega)^{n}$,
and $x_{0}\in\otilc$ such that $\det Df(x_{0})$ is invertible in
$\Rtil$. Then there are sharply open neighborhoods $U$ of $x_{0}$
and $V$ of $f(x_{0})$ such that $f:U\to V$ is a diffeomorphism
in the sense of \cite{AFJ05}. 
\end{theorem}
\begin{proof}
By Thm.~\ref{thm:inclusionCGF}, $f$ can be viewed as an element
of $\gsfo(\otilc,\Rtil^{n})$. Moreover, $\otilc$ is sharply open,
which together with Lemma \ref{lem:invertmatrix} shows that all the
assumptions of Thm.\ \ref{thm:localIFTSharp} are satisfied. We conclude
that $f$ possesses an inverse $f^{-1}$ in $\gsfo(V,U)$ for a suitable
sharp neighborhood $V$ of $f(x_{0})$. Finally, by Lemma \ref{afjcomp},
both $f$ and $f^{-1}$ are infinitely differentiable in the sense
of \cite{AFJ05}. \hspace*{\fill}\qed 
\end{proof}

\section{Global inverse function theorems\label{globsec}}

The aim of the present section is to obtain statements on the global
invertibility of generalized smooth functions. For classical smooth
functions, a number of criteria for global invertibility are known,
and we refer to \cite{DGZ,KP} for an overview.

The following auxilliary result will repeatedly be needed below: 
\begin{lemma}
\label{lem:fromPointwiseToCompact}Let $f\in\gsf(X,Y)$ be defined
by $(f_{\eps})$, where $X\subseteq\rti^{n}$ and $Y\subseteq\rti^{d}$.
Assume that $\emptyset\ne[A_{\eps}]\subseteq X$. Let $b:\R^{d}\ra\R$
be a set-theoretical map such that $\bar{b}:[y_{\eps}]\in Y\mapsto[b(y_{\eps})]\in\rti$
is well-defined (e.g., $b(x)=|x|$). If $f$ satisfies 
\begin{equation}
\forall x\in X:\ \bar{b}\left[f(x)\right]>0,\label{eq:greaterZeroPointwise}
\end{equation}
then 
\begin{enumerate}
\item $\exists q\in\R_{>0}\,\forall^{0}\eps\,\forall x\in A_{\eps}:\,b\left(f_{\eps}(x)\right)>\rho_{\eps}^{q}$. 
\item For all $K\Subset\R^{n}$, if $[K]\subseteq X$ then $\forall^{0}\eps\,\forall x\in K:\,b\left(f_{\eps}(x)\right)>0$. 
\end{enumerate}
\end{lemma}
\begin{proof}
In fact, suppose to the contrary that there was a sequence $(\eps_{k})_{k}\downarrow0$
and a sequence $x_{k}\in A_{\eps_{k}}$ such that $b(f_{\eps_{k}}(x_{k}))\le\rho_{\eps_{k}}^{k}$.
Let $A_{\eps}\ne\emptyset$ for $\eps\le\eps_{0}$, and pick $a_{\eps}\in A_{\eps}$.
Set 
\begin{equation}
x_{\eps}:=\left\{ \begin{array}{rl}
x_{k} & \text{ for }\eps=\eps_{k}\\
a_{\eps} & \text{ otherwise. }
\end{array}\right.\label{xkdef-1}
\end{equation}
It follows that $x:=[x_{\eps}]\in[A_{\eps}]\subseteq X$, and hence
$\bar{b}\left[f(x)\right]>0$ by \eqref{eq:greaterZeroPointwise}.
Therefore, $b\left(f_{\eps_{k}}(x_{k})\right)>\rho_{\eps_{k}}^{p}$
for some $p\in\R_{>0}$ by Lemma \ref{lem:mayer}, and this yields
a contradiction. The second part follows by setting $A_{\eps}=K$
in the first one and by noting that $\rho_{\eps}>0$.\hspace*{\fill}\qed 
\end{proof}

After these preparations, we now turn to generalizing global inverse
function theorems from the smooth setting to GSF. We start with the
one-dimensional case. Here it is well-known that a smooth function
$f:\R\to\R$ is a diffeomorphism onto its image if and only if $|f'(x)|>0$
for all $x\in\R$. It is a diffeomorphism onto $\R$ if in addition
there exists some $r>0$ with $|f'(x)|>r$ for all $x\in\R$. Despite
the fact that $\csp{\R}$ is non-Archimedean, there is a close counterpart
of this result in GSF. 
\begin{theorem}
\label{thm:1D} Let $f\in\gsf(\csp{\R},\csp{\R})$ and suppose that
there exists some $r\in\R_{\ge0}$ such that $|f'(x)|>r$ for all
$x\in\csp{\R}$. Then 
\begin{enumerate}
\item \label{enu:1-dim-f_epsDiffeom}$f$ has a defining net $(\bar{f}_{\eps})$
consisting of diffeomorphisms $\bar{f}_{\eps}:\R\to\R$. 
\item \label{enu:1-dim-fGlobDiffeom}$f$ is a global generalized diffeomorphism
in $\gsf(\csp{\R},f(\csp{\R}))$. 
\item \label{enu:1-dim-fOnto}If $r>0$, then $f(\csp{\R})=\csp{\R}$, so
$f$ is a global generalized diffeomorphism in $\gsf(\csp{\R},\csp{\R})$. 
\end{enumerate}
\end{theorem}
\begin{proof}
Let $(f_{\eps})$ be a defining net for $f$ such that $f_{\eps}\in\cinfty(\R,\R)$
for each $\eps$ (cf.\ Thm.\ \ref{thm:propGSF} \ref{enu:globallyDefNet}).
Since $|f(x)|>0$ for every $x\in\csp{\R}$, Lemma \ref{lem:fromPointwiseToCompact}
implies that for each $n\in\N$ there exists some $\eps_{n}>0$ and
some $q_{n}>0$ such that for each $\eps\in(0,\eps_{n}]$ and each
$x\in[-n,n]$ we have $|f'_{\eps}(x)|>\rho_{\eps}^{q_{n}}$. Clearly
we may suppose that $\eps_{n}\downarrow0$, $q_{n+1}>q_{n}$ for all
$n$ and that $\rho_{\eps}^{q_{n}}<1$. Now for any $n\in\N_{>0}$
let $\phi_{n}:\R\to[0,1]$ be a smooth cut-off function with $\phi_{n}\equiv1$
on $[-(n-1),n-1]$ and $\supp\phi_{n}\sse[-n,n]$. Supposing that
$f_{\eps}'(x)>0$ on $[-n,n]$ (the case $f_{\eps}'(x)<0$ on $[-n,n]$
can be handled analogously), we set 
\begin{align*}
v_{n\eps}(x) & :=f_{\eps}'(x)\phi_{n}(x)+1-\phi_{n}(x)\quad(x\in\R)\\
\bar{f}_{\eps}(x) & :=f_{\eps}(0)+\int_{0}^{x}v_{n\eps}(t)\,dt\quad(x\in\R,\ \eps_{n+1}<\eps\le\eps_{n}),
\end{align*}
and $\bar{f}_{\eps}:=f_{\eps}$ for $\eps\in(\eps_{0},1]$. Then $\bar{f}_{\eps}\in\cinfty(\R,\R)$
for each $\eps$, and for each $x\in\R$ and each $\eps\in(\eps_{n+1},\eps_{n}]$,
we have $\bar{f}'_{\eps}(x)=f_{\eps}'(x)\phi_{n}(x)+1-\phi_{n}(x)>\rho_{\eps}^{q_{n}}$
if and only if $\phi_{n}(x)\cdot\left[1-f'_{\eps}(x)\right]<1-\rho_{\eps}^{q_{n}}$.
The latter inequality holds if $x\notin[-n,n]$ or if $f'_{\eps}(x)\ge1$.
Otherwise, $\phi_{n}(x)\le1<\frac{1-\rho_{\eps}^{q_{n}}}{1-f'_{\eps}(x)}$
because $1>f'_{\eps}(x)>\rho_{\eps}^{q_{n}}$. Any such $\bar{f}_{\eps}$
therefore is a diffeomorphism from $\R$ onto $\R$. Also, $\bar{f}_{\eps}(x)=f_{\eps}(x)$
for all $x\in[-n,n]$ as soon as $\eps\le\eps_{n+1}$. Hence also $(\bar{f}_{\eps})$
is a defining net for $f$. This proves \ref{enu:1-dim-f_epsDiffeom}.

For each $\eps\le\eps_{0}$, let $g_{\eps}$ be the global inverse
of $\bar{f}_{\eps}$. We claim that $g:=[g_{\eps}]$ is a GSF from
$f(\csp{\R})$ onto $\csp{\R}$ that is inverse to $f$. For this
it suffices to show that whenever $y=[(y_{\eps})]\in f(\csp{\R})$,
then for each $k\in\N$, the net $(g_{\eps}^{(k)}(y_{\eps}))$ is
$\rho$-moderate. To see this, it suffices to observe that for $y=f(x)$,
$f$ satisfies the assumptions of the local inverse function theorem
(Thm.\ \ref{thm:localIFTSharp}) at $x$, and so the proof of that result
shows that $g$ is a GSF when restricted to a suitable sharp neighborhood
of $y$. But this in particular entails the desired moderateness property
at $y$, establishing \ref{enu:1-dim-fGlobDiffeom}.

Finally, assume that $r>0$. The same reasoning as in the proof of
\ref{enu:1-dim-f_epsDiffeom} now produces a defining net $(\bar{f}_{\eps})$
with the property that $|\bar{f}_{\eps}'(x)|>r$ for all $\eps\le\eps_{0}$
and all $x\in\R$. Again, each $\bar{f}_{\eps}$ is a diffeomorphism
from $\R$ onto $\R$, and we denote its inverse by $g_{\eps}:\R\to\R$.
Due to \ref{enu:1-dim-fGlobDiffeom} it remains to show that $f:\csp{\R}\to\csp{\R}$
is onto.

To this end, note first that $|g_{\eps}'(y)|<1/r$ for all $\eps\le\eps_{0}$
and all $y\in\R^{n}$. Also, since $f\in\gsf(\csp{\R},\csp{\R})$,
there exists some real number $C>0$ such that $\no{f_{\eps}(0)}\le C$
for $\eps$ small. For such $\eps$ and any $[y_{\eps}]\in\csp{\R}$
we obtain by the mean value theorem 
\begin{equation}
\no{g_{\eps}(y_{\eps})}=\no{g_{\eps}(y_{\eps})-g_{\eps}(f_{\eps}(0))}\le\frac{1}{r}\no{y_{\eps}-f_{\eps}(0)}\le\frac{1}{r}(\no{y_{\eps}}+C),\label{eq:C}
\end{equation}
so that $g_{\eps}(y_{\eps})$ remains in a compact set for $\eps$
small. Based on this observation, the same argument as in \eqref{derest}
shows that, for any $y=[y_{\eps}]\in\csp{\R}$ and any $k\ge1$, $(g_{\eps}^{(k)}(y_{\eps}))$
is moderate, so $(g_{\eps})$ defines a GSF $\csp{\R}\ra\csp{\R}$.
Hence given $y\in\csp{\R}$ it suffices to set $x:=g(y)$ to obtain
$f(x)=y$. \hspace*{\fill}\qed 
\end{proof}

Turning now to the multi-dimensional case, we first consider Hadamard's
global inverse function theorem. For its formulation, recall that
a map between topological spaces is called proper if the inverse image
of any compact subset is again compact. As is easily verified, a continuous
map $\alpha:\R^{n}\to\R^{m}$ is proper if and only if 
\begin{equation}
\no{\alpha(x)}\to\infty\ \text{ as }\ \no{x}\to\infty.\label{properchar}
\end{equation}

\begin{theorem}
\label{thm:hadamard} (Hadamard) A smooth map $f:\R^{n}\to\R^{n}$
is a global diffeomorphism if and only if it is proper and its Jacobian
determinant never vanishes. 
\end{theorem}
For a proof of this result we refer to \cite{Gor}.

The following theorem provides an extension of Thm.~\ref{thm:hadamard}
to the setting of GSF. 
\begin{theorem}
\label{thm:hadgsf}Suppose that $f\in\gsf(\csp{\R^{n}},\csp{\R^{n}})$
possesses a defining net \emph{$f_{\eps}:\R^{n}\ra\R^{n}$} such that: 
\begin{enumerate}
\item \label{enu:diffInvertible}$\forall x\in\R^{n}\,\forall\eps\in I:$
$Df_{\eps}(x)$ is invertible in $L(\R^{n},\R^{n})$, and for each
$x\in\csp{\R^{n}}$, $Df(x)$ is invertible in $L(\rti^{n},\rti^{n})$. 
\item \label{enu:uniformEpsAlpha} There exists some $\eps'\in I$ such
that $\inf_{\eps\in(0,\eps']}\no{f_{\eps}(x)}\to+\infty$ as $|x|\to\infty$. 
\end{enumerate}
Then $f$ is a global generalized diffeomorphism in $\gsf(\csp{\R^{n}},\csp{\R^{n}})$. 
\end{theorem}
\begin{proof}
By Thm.~\ref{thm:hadamard}, each ${f}_{\eps}$ is a global diffeomorphism
$\R^{n}\to\R^{n}$ for each $\eps\le\eps'$ and we denote by $g_{\eps}:\R^{n}\to\R^{n}$
the global inverse of $f_{\eps}$. In order to prove that the net
$(g_{\eps})_{\eps\le\eps'}$ defines a GSF, we first note that, by
\ref{enu:uniformEpsAlpha}, the net $(f_{\eps})_{\eps\le\eps'}$ is
`uniformly proper' in the following sense: Given any $M\in\R_{\ge0}$
there exists some $M'\in\R_{\ge0}$ such that when $\no{x}\ge M'$
then $\forall\eps\le\eps':\ \no{f_{\eps}(x)}\ge M$.

Hence, for any $K\comp\R^{n}$, picking $M>0$ with $K\sse B_{M}(0)$
it follows that $g_{\eps}(K)\sse\overline{B_{M'}(0)}=:K'\comp\R^{n}$
for all $\eps\le\eps'$. Thereby, the net $(g_{\eps})_{\eps\le\eps'}$
maps $\csp{\R^{n}}$ into itself, i.e. 
\begin{equation}
\forall[y_{\eps}]\in\csp{\R^{n}}:\ \left[g_{\eps}(y_{\eps})\right]\in\csp{\R^{n}}.\label{eq:0der}
\end{equation}
Moreover, for each $K\Subset\R^{n}$, assumption \ref{enu:diffInvertible},
Lemma \ref{lem:invertmatrix}, Lemma \ref{lem:mayer} and Lemma \ref{lem:fromPointwiseToCompact}
yield 
\begin{equation}
\exists q\in\R_{>0}\,\forall^{0}\eps\,\forall x\in K:\ |\det Df_{\eps}(x)|>\rho_{\eps}^{q}.\label{detest}
\end{equation}
From \eqref{eq:0der} and \eqref{detest} it follows as in \eqref{derest}
that, for any $y=[y_{\eps}]\in\csp{\R}$ and any $|\beta|\ge1$, $(\partial^{\beta}g_{\eps}(y_{\eps}))$
is moderate, so $g:=[y_{\eps}]\mapsto[g_{\eps}(y_{\eps})]\in\gsf(\csp{\R^{n}},\csp{\R^{n}})$.
Finally, that $g$ is the inverse of $f$ on $\csp{\R^{n}}$ follows
as in Thm.~\ref{thm:1D}.\hspace*{\fill}\qed 
\end{proof}

The next classical inversion theorem we want to adapt to the setting
of generalized smooth functions is the following one: 
\begin{theorem}
\label{thm:hadlev} (Hadamard-Lévy) Let $f:X\to Y$ be a local diffeomorphism
between Banach spaces. Then $f$ is a diffeomorphism if there exists
a continuous non-decreasing function $\beta:\R_{\ge0}\to\R_{>0}$
such that 
\[
\int_{0}^{\infty}\frac{1}{\beta(s)}\,\diff{s}=+\infty,\quad\no{Df(x)^{-1}}\le\beta(\no{x})\quad\forall x\in X.
\]
This holds, in particular, if there exist $a$, $b\in\R_{>0}$ with
$\no{Df(x)^{-1}}\le a+b\no{x}$ for all $x\in X$. 
\end{theorem}
For a proof, see \cite{DGZ}.

We can employ this result to establish the following global inverse
function theorem for GSF. 
\begin{theorem}
\label{thm:hadlevgsf}Suppose that $f\in\gsf(\csp{\R^{n}},\csp{\R^{n}})$
satisfies: 
\begin{enumerate}
\item \label{enu:diffInvertibleHL}$f$ possesses a defining net \emph{$f_{\eps}:\R^{n}\ra\R^{n}$
such that }$\forall x\in\R^{n}\,\forall\eps\in I:$ $Df_{\eps}(x)$
is invertible in $L(\R^{n},\R^{n})$, and for each $x\in\csp{\R^{n}}$,
$Df(x)$ is invertible in $L(\rti^{n},\rti^{n})$. 
\item \label{enu:beta} There exists a net of continuous non-decreasing
functions $\beta_{\eps}:\R_{\ge0}\ra\R_{>0}$ such that $\forall^{0}\eps\,\forall x\in\R^{n}:\ \no{Df_{\eps}(x)^{-1}}\le\beta_{\eps}(\no{x})$
and 
\[
\int_{0}^{\infty}\frac{1}{\beta_{\eps}(s)}\,\diff{s}=+\infty.
\]
\end{enumerate}
Then $f$ is a global generalized diffeomorphism in $\gsf(\csp{\R^{n}},f(\csp{\R^{n}}))$.

If instead of \ref{enu:beta} we make the stronger assumption
\begin{enumerate}[resume]
\item \label{enu:ii'}$\exists C\in\R_{>0}:\ \forall^{0}\eps\,\forall x\in\R^{n}:\ \no{Df_{\eps}(x)^{-1}}\le C,$
\end{enumerate}
\noindent then $f$ is a global generalized diffeomorphism in $\gsf(\csp{\R^{n}},\csp{\R^{n}})$.

In particular, \ref{enu:beta} applies if there exist $a$, $b\in\rti_{>0}$
that are finite (i.e., $a_{\eps}$, $b_{\eps}<R$ for some $R\in\R$
and $\eps$ small) with $\no{Df_{\eps}(x)^{-1}}\le a_{\eps}+b_{\eps}\no{x}$
for $\eps$ small and all $x\in\csp{\R^{n}}$. 

\end{theorem}
\begin{proof}
From \ref{enu:beta} it follows by an $\eps$-wise application of
Thm.~\ref{thm:hadlev} that there exists some $\eps_{0}>0$ such
that each $f_{\eps}$ with $\eps<\eps_{0}$ is a diffeomorphism: $\R^{n}\to\R^{n}$.
We denote by $g_{\eps}$ its inverse. Using assumption \ref{enu:diffInvertibleHL},
it follows exactly as in the proof of Thm.~\ref{thm:1D} \ref{enu:1-dim-fGlobDiffeom}
that $g:=[g_{\eps}]$ is an element of $\gsf(f(\csp{\R^{n}}),\csp{\R^{n}})$
that is inverse to $f$.

Assuming \ref{enu:ii'}, for any $[y_{\eps}]\in\csp{\R^{n}}$ and
$\eps$ small, we have 
\[
\no{Dg_{\eps}(y_{\eps})}=\no{(Df_{\eps}(g_{\eps}(y_{\eps})))^{-1}}\le C,
\]
so the mean value theorem yields 
\begin{equation}
\no{g_{\eps}(y_{\eps})}=\no{g_{\eps}(y_{\eps})-g_{\eps}(f_{\eps}(0))}\le C\no{y_{\eps}-f_{\eps}(0)},\label{eq:Cf0}
\end{equation}
which is uniformly bounded for $\eps$ small since $f\in\gsf(\csp{\R^{n}},\csp{\R^{n}})$.
We conclude that $(g_{\eps})$ satisfies \eqref{eq:0der}. From here,
the proof can be concluded literally as in Thm.~\ref{thm:hadgsf}.
\hspace*{\fill}\qed
\end{proof}

\begin{remark}
By Thm.~\ref{thm:inclusionCGF}, for $\rho=(\eps)$, the space $\gsf(\csp{\R^{n}},\rti^{n})$
can be identified with the special Colombeau algebra $\gs(\R^{n})^{n}$.
In this picture, $\gsf(\csp{\R^{n}},\csp{\R^{n}})$ corresponds to
the space of c-bounded Colombeau generalized functions on $\R^{n}$
(cf.~\cite{Kun02,GKOS}). Therefore, under the further assumption
that $f(\csp{\R^{n}})=\csp{\R^{n}}$, theorems \ref{thm:hadgsf} and
\ref{thm:hadlevgsf} can alternatively be viewed as global inverse
function theorems for c-bounded Colombeau generalized functions. 
\end{remark}

\section{Conclusions}

Once again, we want to underscore that the statement of the local
inverse function theorem \ref{thm:localIFTSharp} is the natural generalization
to GSF of the classical result. Its simplicity relies on the fact
that the sharp topology is the natural one for GSF, as explained above.
This natural setting permits to include examples in our theory that
cannot be incorporated in an approach based purely on Colombeau generalized
functions on classical domains (cf.~Example \ref{exa:examples} and
\cite{EDiss}).

Moreover, as Thm.~\ref{thm:localIFTFermat} shows, the concept of
Fermat topology leads, with comparable simplicity, to sufficient conditions
that guarantee solutions defined on large (non-infinitesimal) neighborhoods.

\vspace{1em}
 \textbf{Acknowledgement:} We are indebted to the referee for several
helpful comments that have substantially improved the results of Section
\ref{globsec}. P.~Giordano has been supported by grants P25116 and
P25311 of the Austrian Science Fund FWF. M.~Kunzinger has been supported
by grants P23714 and P25326 of the Austrian Science Fund FWF.


\begin{thebibliography}{10}
\bibitem{Ab-Ma-Ra}Abraham, R., Marsden, J.E., Ratiu, T.: Manifolds,
Tensors, Analysis and Applications, second ed., Springer-Verlag, New York (1988)

\bibitem{AFJ05} Aragona, J., Fernandez, R., Juriaans, S.O.: A discontinuous
Colombeau differential calculus. Monatsh.~Math. \textbf{144}, 13\textendash 29
(2005)

\bibitem{AFJO12} Aragona, J., Fernandez, R., Juriaans, S.O., Oberguggenberger,
M.: Differential calculus and integration of generalized functions
over membranes. Monatsh.~Math. \textbf{166}, 1\textendash 18 (2012)

\bibitem{Ba-Zo90} Bampi, F., Zordan, C.: Higher order shock waves.
Journal of Applied Mathematics and Physics, \textbf{41}, 12\textendash 19
(1990)

\bibitem{Ber96} Bertrand, D.: Review of `Lectures on differential
Galois theory', by A.~Magid. Bull.~Amer.~Math.~Soc.\ \textbf{33},
289\textendash 294 (1996)

\bibitem{Bo-Sh59}Bogoliubov, N.N., Shirkov, D.V.: The Theory of Quantized
Fields. New York, Interscience (1959)

\bibitem{Cas09}Casale, G.: Morales-Ramis theorems via Malgrange pseudogroup.
Ann.~Inst.~Fourier \textbf{59}, no. 7, ~2593\textendash 2610 (2009)

\bibitem{Col92}Colombeau, J.F.: Multiplication of distributions -
A tool in mathematics, numerical engineering and theoretical Physics.
Springer-Verlag, Berlin Heidelberg (1992)

\bibitem{DGZ}De Marco, G., Gorni, G., Zampieri, G.: Global inversion
of functions: an introduction. NoDEA Nonlinear Differential Equations
Appl.\ \textbf{1} no.\ 3, 229\textendash 248 (1994)

\bibitem{Dir26}Dirac, P.A.M.: The physical interpretation of the
quantum dynamics. Proc.\ R.\ Soc.\ Lond.\ A, \textbf{113} 27,
621\textendash 641 (1926\textendash 27)

\bibitem{EDiss} Erlacher, E.: Local existence results in algebras
of generalized functions. Doctoral thesis, University of Vienna (2007)\\
 \texttt{http://www.mat.univie.ac.at/~diana/uploads/publication47.pdf}

\bibitem{EG:11}Erlacher, E., Grosser, M.: Inversion of a `discontinuous
coordinate transformation' in General Relativity. Appl. Anal. \textbf{90},
1707\textendash 1728 (2011)

\bibitem{EG:13}Erlacher, E., Inversion of Colombeau generalized functions.
Proc.\ Edinb.\ Math.\ Soc.\ \textbf{56}, 469\textendash 500 (2013)

\bibitem{ErlGross}Erlacher, E., Grosser, M.: Ordinary Differential
Equations in Algebras of Generalized Functions. In: Pseudo-Differential
Operators, Generalized Functions and Asymptotics, S. Molahajloo, S.
Pilipovi\'{c}, J. Toft, M. W. Wong eds, Operator Theory: Advances
and Applications Volume 231, 253\textendash 270 (2013)

\bibitem{EsVi12}Estrada, R., Vindas, J.: A general integral. Dissertationes
Math. \textbf{483}, 1\textendash 49 (2012)

\bibitem{CaFe01}Campos Ferreira, J.: On some general notions of superior
limit, inferior limit and value of a distribution at a point. Portugaliae
Math. \textbf{28}\emph{, }139\textendash 158 (2001)

\bibitem{GiKu16}Giordano P., Kunzinger M.: A convenient notion of
compact set for generalized functions. Proc.\ Edinb.\ Math.\ Soc.,
in press. See arXiv 1411.7292.

\bibitem{Gi-Ku-St15}Giordano P., Kunzinger M.: Steinbauer R.: A new
approach to generalized functions for mathematical physics. See http://www.mat.univie.ac.-
at/\~{ }giordap7/GenFunMaps.pdf.

\bibitem{GKV}Giordano, P., Kunzinger, M., Vernaeve, H.: Strongly
internal sets and generalized smooth functions. Journal of Mathematical
Analysis and Applications \textbf{422}, issue 1, 56\textendash 71
(2015)

\bibitem{GiLu15}Giordano, P., Luperi Baglini, L.: Asymptotic gauges:
Generalization of Colombeau type algebras. Math. Nachr.\ \textbf{289},
2-3, 1\textendash 28 (2015)

\bibitem{Gon-Sca56}Gonzales Dominguez, A., Scarfiello, R.: Nota sobre
la formula v.p. $\frac{1}{x}\cdot\delta=-\frac{1}{2}\delta'$. Rev.
Un. Mat,. Argentina \textbf{1}, 53\textendash 67 (1956)

\bibitem{Gor}Gordon, W. B.: On the diffeomorphisms of Euclidean space.
Amer.\ Math.\ Monthly \textbf{79}, 755\textendash 759 (1972)

\bibitem{GMS}Grant, K.D.E., Mayerhofer, E., Steinbauer, R.: The wave
equation on singular space-times. Commun.\ Math.\ Phys.\ \textbf{285},
399\textendash 420 (2009)

\bibitem{GKOS}Grosser, M., Kunzinger, M., Oberguggenberger, M., Steinbauer,
R.: Geometric theory of generalized functions. Kluwer, Dordrecht (2001)

\bibitem{Gsp09}Gsponer, A.: A concise introduction to Colombeau generalized
functions and their applications in classical electrodynamics. European
J. Phys.\ \textbf{30}, no. 1, 109\textendash 126 (2009)

\bibitem{Hair14}Hairer, M.: A theory of regularity structures. Invent.\
Math.\ \textbf{198}, 269\textendash 504 (2014)

\bibitem{HKS} Hörmann, G., Kunzinger, M., Steinbauer, R.: Wave equations
on non-smooth space-times. In: Asymptotic Properties of Solutions
to Hyperbolic Equations, M.~Ruzhansky and J.~Wirth (Eds.) Progress
in Mathematics \textbf{301}. 162\textendash 186, Birkhäuser (2012).

\bibitem{Kan98}Kanwal, R.P.: Generalized functions. Theory and technique.
2nd edition, Birkhäuser (1998)

\bibitem{Kat-Tal12}Katz, M.G., Tall, D.: A Cauchy-Dirac delta function.
Foundations of Science (2012) See http://dx.doi.org/10.1007/s10699-012-9289-4
and http://arxiv.org/abs/1206.0119.

\bibitem{KoKuMO08}Konjik, S., Kunzinger, M., Oberguggenberger, M.:
Foundations of the Calculus of Variations in Generalized Function
Algebras. Acta Applicandae Mathematicae \textbf{103} n.~2, 169\textendash 199
(2008)

\bibitem{KP} Krantz, S. G., Parks, H. R.: The implicit function theorem.
History, theory, and applications. Reprint of the 2003 edition. Modern
Birkhäuser Classics. Birkhäuser/Springer, New York (2013)

\bibitem{Kun02}Kunzinger: M. Generalized functions valued in a smooth
manifold. Monatsh. Math., \textbf{137} 31\textendash 49 (2002)

\bibitem{Lau89}Laugwitz, D.: Definite values of infinite sums: aspects
of the foundations of infinitesimal analysis around 1820. Arch. Hist.
Exact Sci. \textbf{39}, no. 3, 195\textendash 245 (1989)

\bibitem{Lau92}Laugwitz, D.: Early delta functions and the use of
infinitesimals in research. Revue d'histoire des sciences, tome \textbf{45},
1, 115\textendash 128 (1992)

\bibitem{Loj57}\L ojasiewicz, S.: Sur la valeur et la limite d'une
distribution en un point. Studia Math. \textbf{16}, 1\textendash 36
(1957)

\bibitem{Loj58}\L ojasiewicz, S.: Sur la fixation des variables dans
une distribution. Studia Math. \textbf{17}, 1\textendash 64 (1958)

\bibitem{GiLu16}Luperi Baglini, L., Giordano, P.: A fixed point iteration
method for arbitrary generalized ODE. In preparation.

\bibitem{Mal02}Malgrange, B.: On nonlinear differential Galois theory.
Dedicated to the memory of Jacques-Louis Lions.~Chinese Ann.~Math.~Ser.~B
\textbf{23}, no.~2, pp.~219\textendash 226 (2002).

\bibitem{Mal10}Malgrange, B.: Pseudogroupes de Galois et théorie
de Galois différentielle. Prépublications IHES M/10/11 (2010).

\bibitem{May08}Mayerhofer, E.: On Lorentz geometry in algebras of
generalized functions. Proc.\ Roy.\ Soc.\ Edinburgh Sect.\ A \textbf{138},
no. 4, 843\textendash 871 (2008)

\bibitem{Mar68}Marsden, J.E.: Generalized Hamiltonian mechanics a
mathematical exposition of non-smooth dynamical systems and classical
Hamiltonian mechanics. Archive for Rational Mechanics and Analysis
\textbf{28}, n.~5, 323\textendash 361 (1968).

\bibitem{Mar69}Marsden, J. E.: Non smooth geodesic flows and classical
mechanics. Canad. Math. Bull.\ \textbf{12}, 209\textendash 212 (1969)

\bibitem{Mik66}Mikusinski, J.: On the square of the Dirac delta distribution.
Bull. Acad. Polon. Sci. \textbf{14}, 511\textendash 513 (1966)

\bibitem{ObVe08}Oberguggenberger, M., Vernaeve, H.: Internal sets
and internal functions in Colombeau theory. J.~Math.~Anal.~Appl.\ \textbf{341},
649\textendash 659 (2008)

\bibitem{Pee68}Peetre, J.: On the value of a distribution at a point.
Portugaliae Math. \textbf{27}, 149\textendash 159 (1968)

\bibitem{Raj82}Raju, C.K.: Products and compositions with the Dirac
delta function. J.\ Phys.\ A: Math.\ Gen.\ \textbf{15}, 381\textendash 396
(1982)

\bibitem{Rob66}Robinson, A., Non-Standard Analysis. North-Holland,
Amsterdam, (1966)

\bibitem{Rob73}Robinson, A.: Function theory on some nonarchimedean
fields. Amer. Math. Monthly \textbf{80} (6) 87\textendash 109; Part
II: Papers in the Foundations of Mathematics (1973)

\bibitem{Seb54}Sebastião e Silva, J.: Sur une construction axiomatique
de la theorie des distributions. Rev. Fac. Ciências Lisboa, \textbf{2a}
Serie A, 4, pp.~79-186, (1954/55)

\bibitem{vdP-Bre87}van der Pol, B., Bremmer, H.: Operational calculus
(3rd ed.). New York: Chelsea Publishing Co.\ (1987)

\bibitem{VeVi12}Vernaeve, H., Vindas, J.: Characterization of distributions
having a value at a point in the sense of Robinson. J.~Math.~Anal.~Appl~\textbf{396},
371\textendash 374 (2012)
\end{thebibliography}
\end{document}